\documentclass[11pt,bezier]{article}
 \usepackage{amsmath,amssymb,amsfonts,euscript,graphicx,amsthm}
                        \usepackage[all]{xy}
                        \usepackage{multicol}
                        \usepackage{enumerate}

  \evensidemargin =   -  3 cm \topmargin = 1 cm
  \newtheorem{The}{Theorem}[section]
  \newtheorem{Pro}[The]{Proposition}
  \newtheorem{Lem}[The]{Lemma}
  \newtheorem{Cor}[The]{Corollary}
  
  \newtheorem{Defs}[The]{Definitions}
  \newtheorem{Rem}[The]{Remark}
  
  \newtheorem{Examp}[The]{Example}

    \let\oldproofname=\proofname
   \renewcommand{\proofname}{\textit{\rm\bf\oldproofname}}

\title{\bf\Large Virtually Semisimple Modules and a
Generalization of the Wedderburn-Artin Theorem \thanks {The
research of the first author was in part supported by
a grant from IPM (No. 94130413). This research is partially
carried out in the IPM-Isfahan Branch.}
\thanks {2010{ \it Mathematics Subject Classification}. Primary
16D60, 16S50, 13F10 Secondary 16D70. }}

\author{{\bf M. Behboodi$^{{\rm a,b}}$\thanks{Corresponding
author.}, {\bf A. Daneshvar}$^{{\rm a}}$ ~ and~} {\bf M. R.
Vedadi$^{{\rm a}}$}\\
{\small{ $^{{\rm a}}$Department of Mathematical Sciences,
Isfahan University of Technology}}\vspace{-1mm}\\
    {\small{ P.O.Box :  84156-83111,   Isfahan,   Iran}}\\
{\small{ $^{{\rm b}}$School of Mathematics, Institute for
Research in Fundamental Sciences
(IPM)}}\vspace{-1mm}\\ {\small{ P.O.Box : 19395-5746, Tehran,
Iran}}\vspace{-1mm}\\
   {\small{mbehbood@cc.iut.ac.ir}}\vspace{-1mm}\\
   {\small{a.daneshvar@math.iut.ac.ir}}\vspace{-1mm}\\
   {\small{mrvedadi@cc.iut.ac.ir}}}
  \date{}

\begin{document}
  \maketitle

 \begin{abstract}
 \small{ \noindent
 By any measure,  semisimple modules form one of  the most important classes of modules and play a distinguished role in the module theory and its applications. One of the most
fundamental results in this area   is the Wedderburn-Artin theorem. In this paper, we establish natural generalizations of semisimple modules and give a generalization  of the Wedderburn-Artin theorem. We study modules in which every submodule is isomorphic to a direct summand and name them {\it virtually semisimple modules}. A module $_RM$ is called {\it completely virtually semisimple} if each submodules of $M$ is a virtually semisimple module. A ring $R$ is then called {\it left} ({\it completely}) {\it virtually semisimple} if $_RR$ is a left (compleatly) virtually
semisimple $R$-module. Among other things, we give several characterizations of  left (completely) virtually semisimple rings. For instance, it is shown that a ring $R$ is left completely virtually semisimple if and only if
$R \cong \prod _{i=1}^ k M_{n_i}(D_i)$ where $k, n_1, ...,n_k\in \Bbb{N}$ and each $D_i$ is a principal left ideal domain. Moreover, the integers $k,~ n_1, ...,n_k$ and the principal left
ideal domains $D_1, ...,D_k$ are uniquely determined (up to isomorphism) by $R$.  }
   \end{abstract}

 \section{\bf Introduction}
 In the study of $k$-algebras $A$ (associative or non-associative) over a commutative ring $k$, the semisimplicity plays an  important role. It is known that for a  $k$-module $M$, any $A$-module structure corresponds to a $k$-algebra homomorphism $\theta: A \rightarrow$ End$_k(M)$ with $\theta(a) =$ the multiplication map by $a\in A$ and $\ker \theta
$ = Ann$_A(M)$. Furthermore, if $M$ is an $A$-module with $S = $ End$_A(M)$,  then the image of
$\theta$ lies in End$_S(M)$. This shows that if either $k$ is a division ring and $A$ is a simple
algebra or $A$ has a simple faithful module, then  $A$ can be  represented by linear transformations and in  the finite dimension case, $A$ is a subalgebra of an $n$-by-$n$ matrix
algebra over a division ring. The concept of semisimple module is often introduced as a direct sum of simple ones.
 E. Cartan characterized semisimple Lie algebras and shown that every finite-dimensional module over a semisimple
  Lie algebra with zero characteristic is a
 semisimple module (see for example \cite[Page 27]{Shaf}).
Finite  dimensional algebras  are serious examples of Artinian
algebras
 (algebras with descending chain conditions on their left
(right) ideals). They are considered as Artinian rings in the
ring theory.
 Artinian rings with a
 faithful semisimple module are known to  be {\em semisimple} rings
 that form a
fundamental and important class of rings. If $R$ is any ring, an
$R$-module $M$
is said to be {\it simple}  in case $M\neq (0)$
and it has no non-trivial submodules. {\em Semisimple}
$R$-modules are
then considered as direct sums of simple $R$-modules. It is well known that $R$ is a semisimple ring
if and only if the left
(right) $R$-module
$R$ is semisimple if and only if all left (right) $R$-modules
are
semisimple.
As the historical perspective, the fundamental characterization of
finite-dimensional $k$-algebras was originally done by Wedderburn
in his 1907 paper (\cite{Wed}). After that in 1927, E. Artin generalizes the
Wedderburn's theorem for semisimple algebras (\cite{Ar}). In fact, the
Wedderburn-Artin's result is a landmark in the theory of
noncommutative rings. We recall this theorem as follows:

\noindent{\bf Wedderburn-Artin Theorem}: {\em A ring $R$ is
semisimple if
and only if   $R \cong \prod _{i=1}^ k M_{n_i}(D_i)$ where $k, n_1, \ldots,n_k\in\Bbb{N}$ and each $D_i$ is a  division ring.  Moreover, the integers $k, n_1, \ldots,n_k$  and the division rings $D_1,\dots , D_k$  are uniquely determined  (up to a
permutation).}

By the Wedderburn-Artin Theorem, the study of semisimple rings can be reduced to the
study of modules over division rings. We note that a semisimple module is a type of module that can be understood
easily from its parts. More precisely, a module $M$ is semisimple if and only if every
submodule of $M$ is a direct summand. In this paper,  we study modules (resp., rings) in which every submodule (resp., left ideal) is isomorphic to a direct summand.  We will show that the study of
such rings can  be reduced to the study of modules over principal left ideal domains. This gives a
generalization of the Wedderburn-Artin Theorem.

 Throughout this paper, all rings are associative with identity and all modules are unitary. Following \cite{Good2}, we denote
by ${\rm K.dim}(M)$ the {\it Krull dimension} of a module $M$.
If $\alpha\geqslant 0$ is an ordinal number then the module $M$
is said to be $\alpha$-{\it critical} provided
${\rm K.dim}(M) = \alpha$ while ${\rm K.dim}(M/N) <\alpha$ for
all
non-zero submodules $N$ of $M$. A module is called {\it
critical}
 if it is $\alpha$-critical for
some ordinal $\alpha\geqslant 0$.
\begin{Defs}\label{lem:krullR} {\rm
 We say that an $R$-module $M$ is  {\it virtually semisimple}
if each submodule of $M$ is isomorphic to a direct summand of
$M$.
If each submodule of $M$ is a virtually semisimple module, we
call $M$ {\it completely virtually semisimple}.
If $_RR$ is (resp., $R_R$) is a virtually semisimple module, we
then say that $R$ is a {\it left}
(resp., {\it right}) {\it virtually semisimple ring}. A {\it
left} (resp., {\it right})
{\it completely  virtually simple ring} is similarly defined.} \end{Defs}

In Section 2,
   we introduce the fundamental tools of this study and  give
some basic properties of virtually semisimple modules. Among of other things, we show for a non-zero
virtually semisimple
module $_RM$ the following statements are equivalent: {\rm (1)}
$_RM$ is finitely generated; {\rm (2)} $_RM$ is Noetherian; {\rm
(3)} ${\rm u.dim}(_RM) < + \infty$, and
{\rm (4)} $M \cong R/P_1 \oplus \ldots \oplus R/P_n$ where
$n\in\Bbb{N}$,
each $P_i$ is a quasi prime left ideal of $R$ such that $R/P_i$
is a critical Noetherian $R$-module
(here, ${\rm u.dim}(_RM)$ is the uniform dimension of the module
$_RM$ and we say that a left ideal $P$ of a ring $R$
    is {\it quasi prime} if $P\neq R$ and,
for ideals $A$, $B\subseteq R$, $AB\subseteq P \subseteq A\cap
B$
implies that $A\subseteq P$ or $B\subseteq P$) (see Proposition
\ref{pro:property1}). Also, it is shown that a finitely
generated quasi projective 1-epi-retractable $R$-module $M$
is virtually semisimple if and only if End$_R(M) $ is a
semiprime principal left ideal
ring (Theorem \ref{the:end}). An $R$-module $M$ is called
(resp., {\it $n$-epi-retractable}) {\it epi-retractable} if for
every (resp., n-generated)
submodule $N$ of $M$ there exists an epimorphism
$f : M \longrightarrow N$. This concept is studied in \cite
{G-V}.

Section 3 is devoted to study of the structure of left
(completely) virtually semisimple rings.
 We give several characterizations of left virtually semisimple
rings in Theorem \ref{the:charlvss}. We shall give some examples
to show that the (completely) virtually semisimple
are not symmetric properties for a ring, and also completely
virtually
semisimple modules properly lies between the class of semisimple
modules and the class of virtually
semisimple modules; see Examples
\ref{exa:notunique}$\sim$\ref{exa:notsemsimple}.
 While the left virtually
 semisimple is not a Morita invariant ring property, we proved
 that the left completely virtually semisimple is (see
 Proposition \ref{pro:ringmorita}). In Theorem \ref{the:gene},
we will give the following generalization of the
Wedderburn-Artin theorem:

\noindent{\bf A Generalization of the Wedderburn-Artin Theorem}:
{\em A ring $R$ is left completely virtually semisimple if and
only if $R \cong \prod _{i=1}^ k M_{n_i}(D_i)$ where $k, n_1, ...,n_k\in \Bbb{N}$ and each $D_i$ is a principal left ideal domain. Moreover, the integers $k,~ n_1, ...,n_k$ and the principal left
ideal domains $D_1, ...,D_k$ are uniquely determined (up to isomorphism) by $R$.}

Any unexplained terminology and all the basic results on rings
and
modules that are used in the sequel can be found in \cite{Ful}
and
\cite{Lam2}.

\section{\bf Virtually semisimple modules}

The subject of our study in this section is some basic
properties of virtually semisimple modules. We introduce the
fundamental tools of this study for latter uses.

A module $_RM$ is said to be {\it Dedekind finite} if $M = M
\oplus N$ for some $R$-module $N$, then $N = 0$. Let $M$ and $P$
be $R$-modules. We recall that $P$ is {\it $M$-projective} if
every diagram
in $R$-Mod with exact row
 \begin{displaymath}
    \xymatrix{& P\ar[d] & \\
    M \ar[r] & N \ar[r] & 0 }
\end{displaymath}
can be extended commutatively by a morphism $P \longrightarrow
M$. Also, if $P$ is $P$-projective, then $P$ is called {\it
quasi projective}. A direct summand $K$ of $M$ is denoted by $K
\leq^{\oplus} {_RM}$.

\begin{Pro} \label{pro:property0}
The following statements hold.

\noindent {{\rm (i)}} Let $_RM$ be a non-zero virtually
semisimple module. If $M=M_1 \oplus M_2$ is a decomposition
\indent for $_RM$ such that ${Hom_R}(M_1, M_2)=0$ then $M_2$ is
a virtually semisimple $R$-module.\\
\noindent {{\rm (ii}}) If $I$ is an ideal of $R$ with $IM=0$.
Then $M$ is a virtually semisimple ($R/I$)-module if \indent and
only if $_RM$ is virtually semisimple.\\
\noindent {{\rm (iii)}} A module $_RM$ is virtually semisimple
quasi projective if and only if it is an epi- \indent
retractable $R$-module and all of its submodules are
 $M$-projective.\\
\noindent {{\rm (iv)}} Being (completely) virtually semisimple
module is a Morita invariant property.\\
\noindent {{\rm (v)}} Let $_RM$ be virtually semisimple and
$W\leq {_RM}$. If $W$ contains any submodule $K$ of \indent $M$
with $K$ is embedded in $W$, then $_RW$ is virtually semisimple
and there is a direct \indent summand $K$ of $_RM$ such that
${K}\cong {W}$ and $K\oplus K'=W$ for some submodule $K'$. In
\indent particular, if $_RW$ is Dedekind finite, then $W \leq
^{\oplus} M$.
\end{Pro}

\begin{proof}
(i) Let $K \leq M_2$. Since $_RM$ is a virtually semisimple
module, there is a decomposition $N \oplus N'=M$ with $N \cong
K$. It easily seen that $M_1$ is fully invariant submodule of
$M$ because ${\rm Hom_R}(M_1, M_2)=0$. It follows that
$M_1=(N \cap M_1) \oplus (N' \cap M_1)$ and
hence $N \cap M_1=0$ because ${\rm Hom_R}(M_1, N)=0$. Thus $M_1
\subseteq N'$ which implies that $N'=M_1 \oplus (N' \cap M_2)$.
Now we have $M=N \oplus M_1 \oplus (N' \cap M_2)=M_1 \oplus
M_2$. It follows that $N$ and hence $K$ is isomorphic to a
direct summand of $M_2$, as desired.\\
(ii) This is routine.\\
(iii) The necessity is clear by the definition and
\cite[Proposition 18.2]{Wis}. Conversely, assume that $_RM$ is
epi-retractable and every submodule of $M$ is $M$-projective. If
$N\leq M$, then there is a surjective homomorphism $f : M
\longrightarrow N$. Since now $N$ is $M$-projective, ${\rm
Ker}f$ must be a direct summand of $M$, the proof is complete.\\
(iv) It is shown that Morita equivalences preserve monomorphisms
and direct sums (see for instance \cite[\S 21]{Ful}).\\
(v) If $W \leq {_RM}$ as stated in the above, then by our
assumption, $W \cong K$ where $K \leq^{\oplus} {_RM}$ and $K
\subseteq W$.
It follows that $K$ is also a direct summand of $W$. Similarly,
if $N\leq W$ and $V \oplus V'=M$ with $V \cong N$, we can deduce
that $N$ is isomorphic to a direct summand of $W$, that is $_RW$
is virtually semisimple. The last statement is now clear.
\end{proof}

Let $R_1$ and $R_2$ be rings and $T=R_1 \oplus R_2$. It is
well-known that any $T$-module $M$ has the form $M_1 \oplus
M_2$, for some $R_i$-modules $M_i$($i$=1, 2). In fact,
$M=e_1M\oplus e_2M$ where $e_1$ and $e_2$ are central orthogonal
idempotents in $T$ such that $e_1R_2=e_2R_2=0$ and
$e_1+e_2=1_T$. Clearly $e_iM$ is naturally an $R_i$-module (as
well as $T$-module) for $i$=1, 2. This shows that ${\rm
Hom_T}(M_i,M_j)=0$ for $i\neq j$. Thus by Proposition
\ref{pro:property0}, we have the following result:

\begin{Cor}\label{cor:R1R2}
Let $R_i$ (1 $\leqslant$ i $\leqslant$ n) be rings, $T=
\prod_{i=1} ^n R_i$ and $M=M_1 \oplus ... \oplus M_n$ be a
$T$-module where each $M_i$ is an $R_i$-module. Then $_{R_i}
M_i$ is (completely) virtually semisimple if and only if $_TM$
is (completely) virtually semisimple.
\end{Cor}

A non-zero submodule $N$ of $M$ is called {\it essential
submodule} if $N$ has non-zero intersection with every non-zero
submodule of $M$ and denoted by $N\leq_e M$. Each left
$R$-module $M$ has a singular submodule consisting of elements
whose annihilators are essential left ideals in $R$. In set
notation it is usually denoted as ${\rm Z}(M)$ and $M$ is called
a {\it singular} (resp., {\it non-singular}) {\it module} if
${\rm Z}(M)=M$ (resp., ${\rm Z}(M)=0$). Also, direct sum of
simple submodules of $_RM$ is denoted by ${\rm Soc}(_RM)$.
Direct sum of pairwise isomorphic submodules is called {\it
homogenous components}.
The following result shows that the study of virtually
semisimple modules $M$ with Dedekind finite $Z(M)$ reduces to
the study of such modules when they are either singular or
non-singular.

\begin{Pro}\label{pro:Zsoc} Let $R$ be a ring and $M$ be an
$R$-module. Then:

\noindent {\rm (i)} If $M$ is virtually semisimple such that
$Z(M)$ is Dedekind finite. Then $M \cong W \oplus L$ \indent
where $W$ is a singular virtually semisimple $R$-module and $L$
is a non-singular virtually \indent semisimple $R$-module.\\
\noindent {\rm (ii)} If every homogenous components of $Soc(M)$
is finitely generated. Then $M$ is virtually \indent semisimple
if and only if $M\cong W \oplus L$ where $W$ is a semisimple
$R$-module and $L$ is a \indent virtually semisimple $R$-module
with $Soc(L)=0$.
\end{Pro}

\begin{proof}
(i) This is obtained by parts (i) and (v) of Proposition
\ref{pro:property0}.\\
(ii) By hypothesis, ${\rm Soc}(M)$ is Dedekind finite and hence,
$M={\rm Soc}(M) \oplus L$ when $_RM$ is virtually semisimple.
Conversely, assume that $M=W\oplus L$ where $_RW$ is semisimple
and $_RL$ is virtually semisimple with ${\rm Soc}(L)=0$. Note
that $W={\rm Soc}(M)$ and let $N \leq {_RM}$. Then there exists
a submodule $W'$ of $M$ such that ${\rm Soc}(N) \oplus W'=W$.
Since ${\rm Soc}(N)\leq N \leq W \oplus L$, so $N={\rm Soc}(N)
\oplus T$ where $T$ is a submodule of $M$. Since ${\rm
Soc}(M)=W$, so ${\rm Soc}(T)=0$ and hence $T$ is embedded in
$L$.
Now, since $L$ is virtually semisimple, so $T$ is isomorphic to
a direct summand of $L$. It follows that $N$ is isomorphic to a
direct summand of $M$. Thus $M$ is a virtually semisimple
$R$-module.
\end{proof}

The following example shows that the hypothesis of {\it ``$Z(M)$
is Dedekind finite"} in Part (i), and {\it ``every homogenous
components of $Soc(M)$ is finitely generated"} in Part (ii) of
Proposition \ref{pro:Zsoc} can not be relaxed.

\begin{Examp} {\rm
For $R=\mathbb{Z}_4$ and $M={\mathbb{Z}_4 }\oplus
(\bigoplus_{i=1}^\infty{\mathbb{Z}_2})$, we have ${\rm
Z}(_RM)={\rm Soc}(_RM)$ is not a direct summand of $_RM$. Since
$R$ is a commutative Artinian principle ideal ring, every
$R$-modules is a direct sums of cyclic modules (see \cite[Result
1.3]{Beh}). We note that since $M$ is countable, all submodules
of $M$ are also countable. It follows that every submodule of
$M$ is isomorphic to $\mathbb{Z}_4 \oplus (\bigoplus_{j\in
J}{\mathbb{Z}_2})$ or $\bigoplus_{j\in J}{\mathbb{Z}_2}$ where
$J$ is an index set with $|J|\leq \aleph$, i.e., $M$ is a
virtually semisimple $R$-module.}
\end{Examp}

A module is called a {\it uniform module} if the intersection of
any two non-zero submodules is non-zero. The {\it Goldie
dimension} of a module $M$, denoted by {\it $ u.dim(M)$}, is
defined to be $n$ if there exists a
finite set of uniform submodules $U_i$ of $M$ such that
$\bigoplus _{i=1}^n U_i$ is an essential submodule of $M$.
If no such finite set of submodules exists, then ${\rm u.dim}
(M)$ is defined to be infinity
(in the literature, the Goldie dimension of $M$ is also called
the rank, the Goldie rank, the uniform dimension of $M$).

Next we need the following two lemmas.

\begin{Lem}\label{lem:nothekrull}
\textup{(See \cite[Lemmas 15.3 and 15.8]{Good2})} If $_RM$ is a
Noetherian module, then $K.dim(M)$ is defined. If $_RM$ is a
non-zero module with Krull dimension, then $_RM$ has
a critical submodule.
\end{Lem}

\begin{Lem}\label{lem:krullR}
\textup{(See \cite[Lemma 6.2.5]{Mcc})} If $_RM$ is finitely
generated, then $K.dim(M) \leqslant K.dim(_RR)$.
 \end{Lem}

In the following proposition, we investigate some finiteness
conditions of virtually semisimple modules.

\begin{Pro}\label{pro:property1} For a non-zero virtually
semisimple module $M$, the following conditions are equivalent.

   \noindent {\rm (1)} $M$ is finitely generated.\\
   \noindent {\rm (2)} $M$ is Noetherian.\\
   \noindent {\rm (3)} ${\rm u.dim}(M) < + \infty$.\\
\noindent{\rm (4)} $M \cong R/P_1 \oplus \ldots \oplus R/P_n$
where $n\in\Bbb{N}$ and each $P_i$ is a quasi prime left ideal
of $R$ such \indent that $R/P_i$ is a critical Noetherian
$R$-modules.\vspace{2mm}\\
In any of the above cases, $M\cong  N$ for all $N \leq_{e} M$.      \end{Pro}

\begin{proof}
(4) $\Rightarrow$ (1) and (1) $\Rightarrow$ (2) are by the facts
that direct summands and finite direct sums of finitely
generated module are again finitely generated.\\
(2) $\Rightarrow$ (3) is well-known for any module (see for
instance \cite[Corollary 5.18]{Good2}).\\
(3) $\Rightarrow$ (4). Assume that ${\rm u.dim}(M)$ is finite
and we set ${\rm u.dim}(M)=n$ where $n\in\Bbb{N}$. Then there
exist uniform cyclic independent submodules $U_1$,\ldots,$U_n$
of $M$ such that $ U_1 \oplus ...\oplus U_n :=N \leq _e{_RM}$.
Since $M$ is virtually semisimple, so $N\cong K$ where $K \oplus
K'=M$ for some $K' \leq M$. Since ${\rm u.dim}(K)={\rm
u.dim}(M)$, so ${\rm u.dim}(K')=0$ (see \cite[Corollary
5.21]{Good2}). Thus $M\cong N$, which implies that $_RM$ is
finitely generated, hence $_RM$ is Noetherian, as we see in the
proof of (1) $\Rightarrow$ (2).
Now by Lemma \ref{lem:nothekrull}, every non-zero submodule of
$M$ contains a non-zero cyclic critical $R$-submodule. Thus the
condition (3) and our assumption imply that $M$ is isomorphic to
$V_1 \oplus \ldots \oplus V_n$ where each $V_i $ is a critical
Noetherian $R$-module. Now assume that $V \cong R/P$ is a
$\alpha$-critical left $R$-module and $AB\subseteq P \subseteq
A\cap B$ for some ideals $A$, $B$ of $R$. If $ A \nsubseteq P$
and $B \nsubseteq P$, then ${\rm K.dim}(R/A)<\alpha$ and ${\rm
K.dim}(R/B)<\alpha$. On the other hand, $B/P$ is a finitely
generated left $(R/A)$-module and hence by \ref{lem:krullR},
${\rm K.dim}(B/P)\leqslant {\rm K.dim}(R/A)<\alpha$. This
contradicts
${\rm K.dim}(R/P)= {\rm max} \{ {\rm K.dim}(R/B), {\rm
K.dim}(B/P) \}$. Therefor $P$ is quasi prime. The proof is
complete.
 \end{proof}

It is easily to see that if $N\cong M$ for all $N\leq_e {_RM}$,
then $_RM$ is virtually semisimple. The Proposition
\ref{pro:property1} shows that a finitely generated module $_RM$
is virtually semisimple if and only if $N\cong M$ for all
$N\leq_e {_RM}$. Thus in this case, $_RM$ is {\it essentially
compressible} in the sense of \cite{S-V} (i.e.,
$M\hookrightarrow N$ for all $N\leq _e {_RM}$). Essentially
compressible modules are {\it weakly compressible} in the sense
of \cite{Zel} (i.e., $N{\rm Hom_R}(M,N) \neq 0$ for any
submodule $0\neq N\leq{_RM}$). We state below some results
related to these concepts and then apply them to investigating
the endomorphism ring of a virtually semisimple module. First we
need the following proposition from several articles. We recall
that a ring $R$ is left hereditary if and only if every left
ideal is projective.

\begin{Pro}\label{pro:niaz}
The following statements hold.

\noindent {\rm (i)} Any essential compressible module is weakly
compressible.\\
\noindent {\rm (ii)} If $_RM$ is a non-zero quasi projective
1-epi-retractable, then ${\rm End_R}(M)$ is a principal \indent
left ideal ring if and only if $_RM$ is epi-retractable. In
particular, $_RR^{(n)}$ is epi-retractable \indent if and only
if $M_n(R)$ is a principal left ideal ring.\\
\noindent {\rm (iii)} If $_RM$ is a quasi projective
retractable, then ${\rm End_R}(M)$ is a semiprime ring if and
only \indent if $_RM$ is weakly compressible.\\
\noindent {\rm (iv)} Every semiprime principal left ideal ring
is a left hereditary ring.
\end{Pro}

\begin{proof}
The part (i) is by Theorems 3.1, 2.2 of \cite{S-V}. (ii) is
obtained by \cite[Theorem 2.2]{G-V}. The part (iii) is
\cite[Theorem 2.6]{H-V} and (iv) is Lemma 4 of \cite{Cam-Coz}.
\end{proof}

Let $R$ be a ring and $M$ be a left $R$-module. If $X$ is an
element or a subset of
$M$, we define the {\it annihilator} of $X$ in $R$ by ${\rm
Ann}_R(X) = \{r \in R~|~rX = (0)\}$. In the case $R$ is
non-commutative and $X$ is an element or a subset of an $R$, we
define the {\it left annihilator} of $X$ in $R$ by ${\rm
l.Ann}_R(X) = \{r \in R~|~rX = (0)\}$ and the {\it right
annihilator} of $X$ in $R$ by ${\rm r.Ann}_R(X) = \{r \in R~|~Xr
= (0)\}$.

 \begin{The}\label{the:end}
Let  $M$ be a quasi projective finitely generated
1-epi-retractable $R$-module.
Then $M$ is virtually semisimple if and only if ${ End}(_RM) $
is a semiprime principal left ideal ring.
 \end{The}

\begin{proof}
Set $S:={\rm End}(_RM)$ and then we apply Proposition
\ref{pro:niaz}.\\
  ($\Rightarrow$). By Proposition \ref{pro:niaz}(ii),
$S$ is a principal left ideal ring. To show that $S$ is also
semiprime, we note that since $_RM$ is virtually semisimple, it
is essentially compressible. Thus by Proposition
\ref{pro:niaz}(i), $_RM$ is weakly compressible and so by
Proposition \ref{pro:niaz}(iii), $S$ is a semiprime ring.\\
($\Leftarrow$). Assume that $S$ is a semiprime principal left
ideal ring. To show that $_RM$ is virtually semisimple, let
$N\leq {_RM}$. Again by Proposition \ref{pro:niaz}(ii), $_RM$ is
epi-retractable and hence there exists a surjective
$R$-homomorphism $f : M \longrightarrow N$. By first isomorphism
theorem, it is enough to show that ${\rm Ker}(f) \leq^{\oplus}
{_RM}$. Note that ${\rm Hom_R}(M, {\rm Ker}(f))={\rm
l.Ann_S}(f)$. Assume that $\varphi : S \longrightarrow Sf$ with
$\varphi(g)=gf$. Since now $_S(Sf)$ is projective by Proposition
\ref{pro:niaz}(iv), so the ${\rm l.Ann_S}(f)= {\rm Ker}
(\varphi)$ is a direct summand of $S$ and hence ${\rm
l.Ann_S}(f)=Se$ for some $e^2=e \in S$. Clearly, ${\rm Im}(e)
\subseteq {\rm Ker}(f)$ and so by the epi-retractable condition,
there exists a surjective homomorphism $h : M \longrightarrow
{\rm Ker}(f)$. So ${\rm Ker}(f) = {\rm Im}(h)$ for some $h \in
S$. Thus $h \in {\rm Hom_R}(M, {\rm Ker}(f))=Se$, which implies
that $h(1-e)=0$. It follows that $h=he$ and hence ${\rm Ker}(f)
={\rm Im}(he) \subseteq {\rm Im}(e)$. This shows that ${\rm
Im}(e) = {\rm Ker}(f)$. It is easily see that ${\rm Im}(e)$ is a
direct summand of $M$ and the proof is complete.
 \end{proof}

\section{\bf Structure of  left virtually semisimple rings}

In this section, we investigate the structure of (completely)
left virtually semisimple rings and
  give  a generalization of Wedderburn-Artin theorem.
Meanwhile, we give an example to show that the left/right
distinction cannot be
removed and also we provide an example of left virtually
semisimple ring which is not completely.
We shall first note that every (right) left virtually semisimple
is principal (right) left ideal ring.
Moreover, we have the following result duo to A. Goldie.

\begin{Lem}\label{lem :goldie}
\textup {(See \cite[Theorem A and B]{Gol})} A semiprime
principal left ideal ring is
finite direct product of matrix rings over let Noetherian
domains.
\end{Lem}

The following lemma is also needed.

\begin{Lem}\label{lem:free}
\textup{\cite[Proposition 2.5]{G-V}} Let $R$ be a left
hereditary ring. Then every left free $R$-module is
epi-retractable if and only if $R$ is a principal left ideal
ring.
\end{Lem}

We investigate below the class of left (completely) virtually
semisimple rings. We should point out that  every set can be
well-ordered (see for instance \cite{Hal}).

 \begin{Pro}\label{pro:ringmorita}
 The following statements hold.\\
\noindent {{\rm (i)}} A ring $R$ is left (completely) virtually
semisimple if and only if every (projective) free \indent left
$R$-module is (completely) left virtually semisimple.\\
\noindent {{\rm (ii)}} Let $R$ be a ring Morita equivalent to a
ring $S$. Then $R$ is a left completely virtually \indent
semisimple if and only if
   $S$ is so.\\
\noindent {{\rm (iii)}} The class of left virtually semisimple
(resp,. left completely virtually semisimple)
    rings \indent is closed under finite direct products.\\
{\rm (iv)} Let $R$ be a left completely virtually semisimple
ring.
Then for any semisimple $R$- \indent module $N$ and projective
$R$-module $P$
with ${\rm Soc}(_RP)=0$, $N\oplus P$ is a completely virtually
\indent semisimple $R$-module.
  \end{Pro}

  \begin{proof}
(i) One direction is clear. In view of Lemma \ref{lem:free} and
Proposition \ref{pro:property0}(iii), we shall prove the left
completely virtually semisimple case. Let $F =\bigoplus_{\alpha
\in \Omega} Re_\alpha$ be a free $R$-module with basis
$\{e_\alpha \} _{\alpha \in \Omega}$ and $P\leq {_RF}$. As the
proof of Kaplansky's theorem \cite[Theorem 2.24]{Lam1},
we fix a well-ordering $`` <"$ on the indexing set $\Omega$. For
any $\alpha \in \Omega$, let
$F_\alpha$ (resp., $G_\alpha$) be the span of the $e_\beta$'s
with $\beta\leqslant \alpha$ (resp., $\beta < \alpha$). Then
each
$a \in P \cap F_\alpha$ has a unique decomposition
$a=b+re_\alpha$ with $b \in G_\alpha$ and $r \in R$.
The mapping $\varphi_\alpha : a \mapsto r$ maps $P \cap
F_\alpha$ onto a left ideal $U_\alpha$ with kernel $P \cap
G_\alpha$.
By Proposition \ref{pro:property0}(iii), $R$ is left hereditary
and so $_RU_\alpha$ is projective. Thus $\varphi_\alpha$ splits,
so we have
$$P \cap F_\alpha=(P \cap G_\alpha) \oplus A_\alpha$$
for some submodule $A_\alpha$ of $P \cap F_\alpha$ isomorphic to
$U_\alpha$. It can be checked that $P=\bigoplus_{\alpha \in
\Omega} A_\alpha$.
Hence $P\cong \bigoplus _{\alpha \in \Omega} U_\alpha$ where
$U_\alpha=\varphi _\alpha ( P \cap F_\alpha )$. It follows that
if $Q\leq {_RP}$, then $Q\cong \bigoplus_{\alpha \in
\Omega}V_\alpha$ where $V_\alpha \subseteq U_\alpha$. Since $R$
is left completely virtually semisimple, so $_RU_\alpha$ is
virtually semisimple for any $\alpha \in \Omega$. Hence
$V_\alpha$ is isomorphic to a direct summand of $U_\alpha$
($\alpha \in \Omega$). This shows that $_RP$ is virtually
semisimple. The proof is complete.\\
(ii) It follows by (i) (since virtually semisimplity and
projectivity conditions are Morita invariants).\\
     (iii) is by Corollary \ref{cor:R1R2}.\\
     (iv) Assume that $K\leq N\oplus P$. We shall show
that $K$ is a virtually semisimple $R$-module. Since ${\rm
Soc}(_RP)=0$,
we have ${\rm Soc}(_R(N\oplus P))=N$ and hence ${\rm Soc}(_RK)$
is a direct summand of $K$.
Thus there is a submodule $K'$ such that $K={\rm Soc}(_RK)
\oplus K'$. Clearly, $K'\cap N=0$
and so $K'$ can be embedded in $_RP$. Now assume that $W\leq K$.
By a similar argument,
we have $W={\rm Soc}(_RW) \oplus W'$ where $W'$ embeds in $K'$.
By part (i), $K'$ is
virtually semisimple $R$-module. Therefore, $W'$ is isomorphic
to a direct summand of $K'$,
     proving that $K$ is a virtually semisimple $R$-module.
     \end{proof}

Several characterizations of left virtually semisimple rings are given below.

 \begin{The}\label{the:charlvss}
 The following statements are equivalent for a ring $R$.\\
\noindent {{\rm (1)}} $R$ is a left virtually semisimple ring.\\\noindent {\rm (2)} $R$ is a
semiprime principal left ideal
ring.\\
\noindent {\rm (3)} $R$ is a left hereditary and principal left
ideal ring.\\
\noindent {\rm (4)} $R \cong \prod_{i=1}^{k} M_{n_i}(D_i)$ where
each $D_i$ is a
(Noetherian) domain and every $M_{n_i}(D_i)$ is a \indent
principal left ideal ring.
   \end{The}

\begin{proof}
(1) $\Rightarrow$ (2). Assume that $R$ is a left virtually
semisimple ring. Then it is clear that $R$ is a principal left
ideal ring. We will to show $R$ is semiprime. Assume that
$I^2=0$ where $I$ is an ideal of $R$ and $L := {\rm r.Ann_R}
(I)$. Then $I \subseteq L$ and for each $0\neq s \in R$, since
$I(Is)=I^2s=0$, we conclude that either $s \in L$ or $Is
\subseteq L$.
This shows that $L$ is an essential left ideal of $R$. Thus $_RR
\stackrel {\varphi} \cong _RL$ by Proposition
\ref{pro:property1}. We have $\varphi (I)=\varphi (IR)=I \varphi
(R)=IL=0$. Thus $I=0$ because $\varphi$ is a monomorphism.
Therefore, $R$ is a semiprime principal left ideal ring.\\
(2) $\Rightarrow$ (3) is by Proposition \ref{pro:niaz}(iv).\\(3) $\Rightarrow$ (1). Since $R$ is a
principal left ideal ring,
so $_RR$ is epi-retractable. Thus by Proposition
\ref{pro:property0}(iii), $R$ is a left virtually semisimple
ring.\\
    (2) $\Rightarrow$ (4) is  by Lemma \ref{lem :goldie}.\\
    (4) $\Rightarrow$ (2) is by the fact that  the class of
semiprime principal left ideal rings is closed under finite
direct products.
\end{proof}

\begin{Rem}\label{uniquness}
The integers $k$ and $n_1, \ldots, n_k$ in the Theorem
\ref{the:charlvss} are uniquely determined by $R$ because of thefollowing lemma.

\begin{Lem}\label{lem:DD'}
Let $\{D_i \} _{i=1}^k$ and $\{D'_j \} _{j=1}^r$ be two families
of left Noetherian domains such that $\prod _{i=1}^k
M_{n_i}(D_i)\cong \prod_{j=1}^rM_{k_j}(D'_j)$ as ring. Then
$r=k$
and there exists a permutation $\xi$ on set $\{1,\ldots,r\}$
such
that for each $i \in \{1,\ldots,r\}$,  $n_i=k_{\xi(i)}$.
\end{Lem}

\begin{proof}
Let $R=\prod_{i}M_{k_i}(D_i)$. Thus $R$ is a semiprime left
Noetherian. Let $Q=Q(R)$
be the classical left quotient ring of $R$. By our assumption,
we can conclude
$Q=\prod_{i}M_{n_i}(Q_i)\cong \prod_{j}M_{k_j}(Q'_j)$ where
$Q_i$ and $Q'_j$ are
division rings. Hence the result is obtained by Wedderburn-Artin
Theorem \cite[Theorem 3.5]{Lam2}.
\end{proof}

\end{Rem}

In view of the above characterization \ref{the:charlvss}(iv) of
a left virtually semisimple ring, the
 following natural question arises:
``{\it Are the domains in Theorem \ref{the:charlvss}(iv) unique
up to isomorphism?}''.
 The following example shows that the answer is negative in
 general.\\

 \begin{Examp}\label{exa:notunique}
{\rm Let $D=A_1(F)$, the first Weyl algebra over a field $F$
with characteristic zero.
It is known that $D$ is a simple Noetherian domain (see for
instance \cite[Chapter 1, \S 3, Theorem 1.3.5]{Mcc}). If $I$ is
a non-zero left ideal of $A_1(F)$ then by \cite[Theorem
3(i)]{Web}, $I\oplus I \cong D\oplus D$ and so $M_2(D)\cong
M_2({\rm End}_{D}(I))$. While by \cite[Proposition 1]{Smi}, $D$
is not isomorphic to the ring ${\rm End}_{D}(I)$.}
 \end{Examp}

In the sequel, we show that the answer of the above question is
positive when $R$ is a left completely virtually semisimple
ring. In this case, $R$ is determined by a set of matrix rings
over principal left ideal domains and the size of matrix rings.
Firstly, we give an example to show that there exists a left
virtually semisimple ring which is not a left completely
virtually semisimple ring.

 \begin{Examp}\label{exa:cnistlhast}
{\rm Let $R=A_1(F)$, the first Weyl algebra over a field $F$
with characteristic zero, and let $S=M_2(R)$. By \cite[Example
7.11.8]{Mcc}, $R$ is not a (left) principal ideal domain and
hence $S$ can not be a left completely virtually semisimple by
Proposition \ref{pro:ringmorita}(ii) while $S$ is a semiprime
principal ideal ring (virtually semisimple ring) by (\cite
[Chapter 7, \S 11, Corollary 7.11.7]{Mcc}).}
 \end{Examp}

The following example shows that for a ring the (completely)
virtually semisimple property is not symmetric.

 \begin{Examp}\label{exa:notsymmetric}
{\rm In \cite[Section 4]{Cohn}, it is given an example of
principal right ideal domain $R$ which is not left hereditary.
Thus by Theorem \ref{the:charlvss}, we deduce that $R$ is right
(completely) virtually semisimple which is not left virtually
semisimple.}
 \end{Examp}

 The following  provide an example
of a simple ring $R$ such that $R$ is a left and a right
(completely) virtually semisimple ring,
 but it is not semisimple.

\begin{Examp}\label{exa:notsemsimple}
 {\rm Let $A$ be a field and $\varphi : A \longrightarrow A$
a ring automorphism such that $\varphi^n \neq 1$ for every
natural number $n$.
  If $R=A[x,x^{-1},\varphi]$,
  then by \cite[Proposition 4.7]{Tug},
$R$ is a simple principal left and right ideal domain that is
not a division ring.
But it is clear that $R$ is a left and right (completely)
virtually semisimple ring.}
\end{Examp}

A ring $R$ is said to be $M_n$-{\it unique} if, for any ring
$S$, $M_n(R) \cong M_n(S)$ implies that $R\cong S$ (see for
instance \cite[\S 17.C]{Lam1}). Let $\mathcal{C}$ be a class of
$R$-modules. Following \cite[\S 17C]{Lam1}, we say that
$\mathcal{C}$ satisfies {\it weak} $n$-{\it cancellation} if,
for any $P$, $Q$ in $\mathcal{C}$, the condition $P^{(n)} \cong
Q^{(n)}$ implies that ${\rm End_R}(P) \cong {\rm End_R}(Q)$. An
$R$-module $P$ is said to be a {\it generator} for $R$-Mod if
$R$ is a direct summand $\bigoplus_{\lambda\in \Lambda} P$ for
some finite index set $\Lambda$.

We need the following lemmas.

\begin{Lem}\label{lem:progen}
\textup{(See \cite[Theorem 17.29]{Lam1})} For any ring S, and
any given integer $n \geq 1$, the following two
statements are equivalent.\\
\noindent {{\rm (1)}} Any ring $T$ Morita-equivalent to $S$ is
$M_n$-unique.\\
\noindent {{\rm (2)}} The class of finitely generated projective
generators in $S$-Mod satisfies the weak $n$- \indent
cancelation property.
\end{Lem}

\begin{Lem}\label{lem:kaplan}
\textup{(Kaplansky's Theorem \cite[ Theorem 2.24]{Lam1})}. Let
$R$ be a left hereditary ring. Then every submodule $P$ of a
free left $R$-module $F$ is isomorphic to a direct sum of left
ideals of $R$.
\end{Lem}

We are now in a position to prove the following generalization
of
the Wedderburn-Artin theorem.

\begin{The}\label{the:gene} Let $R$ be a ring. Then  $R$ is a left completely virtually
semisimple ring if and only if  $R \cong \prod _{i=1}^ k M_{n_i}(D_i)$ where $k, n_1, ...,n_k\in \Bbb{N}$ and each $D_i$ is a principal left ideal domain. Moreover, the integers $k,~ n_1, ...,n_k$ and the principal left
ideal domains $D_1, ...,D_k$ are uniquely determined (up to isomorphism) by $R$.
\end{The}

\begin{proof} $(\Rightarrow)$.
Assume that $R$ is a left completely virtually
semisimple ring. Then by Theorem \ref{the:charlvss}(iv), $R
\cong \prod_{i=1}^{k} M_{n_i}(D_i)$ where each $D_i$ is a domain
and every $M_{n_i}(D_i)$ is a principal left ideal ring. Since
$R$ is a left complectly virtually semisimple ring, so
$M_{n_i}(D_i)$ is a left completely virtually semisimple ring for each $i$ $(1\leqslant i \leqslant k)$. By Proposition
\ref{pro:ringmorita}(ii), each $D_i$ is left completely
virtually semisimple,  i.e., each $D_i$ is a principal left ideal
domain. Hence in view of Lemma \ref{lem:progen} and Theorem
\ref{the:charlvss}, in order to prove that each $D_i$ is
uniquely determined, it is enough to checked that any principal
left ideal domain $D$ satisfies the condition(ii) of Lemma
\ref{lem:progen}.

Suppose that $P$ and $Q$ are finitely generated generators in
$D$-Mod and $P^{(n)}\cong Q^{(n)}$ for $n\geqslant 1$. By Lemma
\ref{lem:kaplan}, every projective $D$-module is free. Thus
$P\cong D^{(r)}$ and $Q\cong D^{(s)}$ for suitable integer
numbers $r$, $s\geq 1$. Therefore, $D^{(nr)}\cong D^{(ns)}$ and
since $D$ is left Noetherian, so by the invariant basis number
property on $D$ (see for instance \cite[Page 17]{Lam1}), we
conclude that $r=s$ and hence $_DP\cong { _DQ}$.\\
$(\Leftarrow)$.   Clearly every principal left ideal
domains is completely virtually semisimple. Thus the implication
is obtained by Proposition \ref{pro:property0}.
\end{proof}

 \end{document}